\documentclass[10pt,bezier]{article}
\usepackage{amsmath,amssymb,amsfonts,graphicx,caption,subcaption}
\usepackage[colorlinks,linkcolor=red,citecolor=blue]{hyperref}

\newtheorem{prethm}{{\bf Theorem}}[section]
\newenvironment{thm}{\begin{prethm}{\hspace{-0.5
em}{\bf.}}}{\end{prethm}}
\newtheorem{prepro}{{\bf Theorem}}

\newtheorem{precor}[prethm]{{\bf Corollary}}
\newenvironment{cor}{\begin{precor}{\hspace{-0.5
em}{\bf.}}}{\end{precor}}
\newtheorem{preconj}[prethm]{{\bf Conjecture}}

\newtheorem{preremark}[prethm]{{\bf Remark}}
\newenvironment{remark}{\begin{preremark}\em{\hspace{-0.5
em}{\bf.}}}{\end{preremark}}
\newtheorem{prelem}[prethm]{{\bf Lemma}}
\newenvironment{lem}{\begin{prelem}{\hspace{-0.5
em}{\bf.}}}{\end{prelem}}
\newtheorem{preque}[prethm]{{\bf Question}}

\newtheorem{preobserv}[prethm]{{\bf Observation}}

\newtheorem{predef}[prethm]{{\bf Definition}}

\newtheorem{preproposition}[prethm]{{\bf Proposition}}
\newenvironment{proposition}{\begin{preproposition}{\hspace{-0.5
em}{\bf.}}}{\end{preproposition}}

\newtheorem{preproof}{{\bf Proof.}}
\newtheorem{preprooff}{{\bf Proof}}

\newenvironment{proof}[1]{\begin{preproof}{\rm
#1}\hfill{$\Box$}}{\end{preproof}}

\newtheorem{preproofs}{{\bf The second proof of }}

\newtheorem{preprooft}{{\bf Third proof of }}

\newtheorem{preproofF}{{\bf Proof of}}

\title{\bf\Large 
Packing  spanning   rigid  subgraphs with restricted degrees
}
\author{{\normalsize{\sc Morteza Hasanvand${}$} }\vspace{3mm}
\\{\footnotesize{${}$\it Department of Mathematical
 Sciences, Sharif
University of Technology, Tehran, Iran}}
{\footnotesize{}}\\{\footnotesize{   $\mathsf{hasanvand@alum.sharif.edu  }$ }}}

\date{}
\def\rank   {\text{rank}}
\begin{document}
\maketitle
\begin{abstract}{
Let $G$ be a graph  and let $l$ be an  integer-valued  function on subsets of  $V(G)$. The graph $G$ is said to be  $l$-partition-connected, if for every partition $P$ of $V(G)$, $e_G(P)\ge \sum_{A\in P} l(A)-l(V(G))$, where $e_G(P)$ denotes the number of edges of $G$ joining different parts of $P$.  We say that $G$ is $l$-rigid,  if it contains a spanning $l$-partition-connected subgraph $H$ with $|E(H)|=\sum_{v\in V(H)} l(v)-l(V(H))$. In this paper, we   investigate  decomposition of  graphs into  spanning partition-connected and spanning rigid subgraphs. As a consequence, we improve a recent result due to Gu (2017) by  proving that every $(4kp-2p+2m)$-connected graph $G$ with  $k\ge 2$ has a spanning subgraph $H$ containing  a  packing  of $m$ spanning  trees and $p$ spanning  $(2k-1)$-edge-connected  subgraphs $H_1,\ldots, H_p$ such that for each vertex $v$, every  $H_i-v$ remains $(k-1)$-edge-connected and also  $d_H(v)\le \lceil \frac{d_G(v)}{2}\rceil +2kp-p+m$. From this result,  we  refine  a  result on arc-connected orientations of graphs.
} {\small
}
\end{abstract}
%
%
%
%
%
%
%
%
%
%
%
%
%
\section{Introduction}
In this article, all graphs have  no  loop, but  multiple  edges are allowed and a simple graph is a graph without multiple edges.
 Let $G$ be a graph. 
The vertex set,  the edge set, and the minimum degree  of $G$ are denoted by $V(G)$, $E(G)$, and $\delta(G)$, respectively. 
The degree $d_G(v)$ of a vertex $v$ is the number of edges of $G$ incident to $v$.
For a set $X\subseteq V(G)$, we denote by $G[X]$ the induced  subgraph of $G$  with the vertex set $X$  containing
precisely those edges 
of $G$  whose ends lie in~$X$.
Let $A$ and $B$ be two subsets of $V(G)$.
This pair  is said to be {\bf intersecting}, if $A\cap B \neq \emptyset$.
Let $l$ be a real function on subsets of $V(G)$ with $l(\emptyset) =0$.
For notational simplicity, we write $l(G)$ for $l(V(G))$ and write $l(v)$ for $l(\{v\})$.
The function $l$ is said to be {\bf  supermodular}, if for all  vertex sets $A$ and $B$,
$l(A\cap B)+l(A\cup B)\ge l(A)+l(B)$.
Likewise, $l$ is said to be  $c$-intersecting supermodular,
 if for all  vertex sets $A$ and $B$  with $|A\cap B|\ge c$, the above-mentioned inequality holds.
When $c=1$, the set function $l$ is said to be  {\bf intersecting supermodular}.
The set function $l$ is called  (i) {\bf nonincreasing}, if  $l(A)\ge l(B)$, for all nonempty vertex sets $A,B$ with $A\subseteq B$,
(ii) {\bf subadditive}, if  $l(A) +l(B)\ge l(A\cup B)$, for any two disjoint   vertex sets $A$ and $B$,
and also  is called (iii) {\bf weakly subadditive}, if  $\sum_{v\in A}l(v)\ge l(A)$, for all  vertex sets $A$.
Note that several  results of  this paper can be hold for real functions $l$ such that $\sum_{v\in A}l(v)-l(A)$ is integer for every vertex set $A$.
For clarity of presentation, we  will  assume that $l$ is  integer-valued.
The graph $G$ is said to be {\bf $l$-edge-connected}, if for all nonempty proper vertex sets $A$,
 $d_G(A) \ge l(A)$, where
$d_G(A)$ denotes the number of edges of $G$ with exactly one end in $A$.
Likewise, the graph $G$ is called {\bf $l$-partition-connected}, 
if for every partition $P$ of $V(G)$,
$e_G(P)\ge \sum_{A\in P}l(A)-l(G),$
where $e_G(P)$  denotes the number of edges of $G$ joining different parts of $P$.
When $\mathcal{P}$ is an arbitrary collection of subsets of $V(G)$,
we denote by $e_G(\mathcal{P})$ the number of edges $e$ of $G$ such that  there is no a vertex set $A$ in $\mathcal{P}$ including both ends of $e$.
We say that a spanning  subgraph $F$  is {\bf $l$-sparse}, if for all vertex sets $A$,
$e_F(A)\le \sum_{v\in A} l(v)-l(A)$, where 
 $e_F(A)$ denotes  the number of edges of $F$ with both ends in $A$. 
Clearly, $1$-sparse graphs are forests.
Note that  all    maximal  spanning $l$-sparse  subgraphs  of $G$ form  the bases of a matroid, 
when $l$ is a $2$-intersecting supermodular weakly subadditive integer-valued  function on subsets of $V(G)$,
 see~\cite{MR0270945}.
Some basic tools in this paper for working with sparse  graphs can be obtained using matroid theory.
We say  that  $G$ is {\bf $l$-rigid}, 
if  it contains a spanning $l$-sparse subgraph $F$ with $|E(F)|=\sum_{v\in V(F)}l(v)-l(F)$.
It is easy to check that an $l$-rigid graph is also $l$-partition-connected.
It was shown that the converse is true, 
when $l$ is an intersecting supermodular weakly subadditive integer-valued function on subsets of $V(G)$~\cite{P}.
For convenience, we write the term `{\bf $k$-rigid}'  for  $l$-rigid, 
where $k$ is an integer and $l=l_{k,2k-1}$ 
where   $l_{m,n}$ denotes the set function that is $m$ on the vertices
 and is $n$  on the vertex sets with at least two vertices.
We say that the  graph $G$ is {\bf $\ell$-weakly $l$-connected},
if for any two disjoint vertex  sets $A$ and $B$ with $A\neq \emptyset$ and $A\cup B\subsetneq V(G)$, 
$d_{G-B}(A) \ge l(A\cup B)-\sum_{v\in B} \ell(v),$
where $\ell$  is a real function on subsets of $V(G)$.
When $G$ is $1$-weakly $l$-connected,  $G$  is said to be  {\bf $l$-connected}.
For every vertex set $A$ of a directed graph $G$,  we  denote  by $d^-_G(A)$  the number of edges   entering $A$ and denote by  $d^+_G(A)$ the number of edges  leaving  $A$.
An orientation of $G$ is called {\bf  $l$-arc-connected}, if for every vertex set $A$, $d_G^-(A)\ge l(A)$.
Likewise, an orientation of $G$ is called {\bf $r$-rooted $l$-arc-connected},
 if for every vertex set $A$, $d_G^-(A)\ge l(A)-\sum_{v\in A}r(v)$, where
 $r$ is a nonnegative integer-valued function on $V(G)$ with $l(G)=\sum_{v\in V(G)}r(v)$.
An orientation of $G$ is said to be {\bf smooth},
 if  for each vertex $v$,  $|d^+_G(v)-d^-_G(v)|\le 1$.
A {\bf packing} refers to a collection  of pairwise
edge-disjoint subgraphs.
Throughout this article, all set functions are zero on the empty set,  all  variables $k$, $p$, and $m$ are integer and nonnegative ($k$ is positive), unless otherwise  stated.
%
%
%
%
%
%
%
%
%
%
%
%
%

In 1961 Nash-Williams and Tutte obtained  a necessary and sufficient condition  for a graph to have  $m$ edge-disjoint spanning trees
 which contains the following result as a corollary.
\begin{thm}{\rm (\cite{MR0133253, MR0140438})}
{Every $2m$-edge-connected graph contains $m$ edge-disjoint spanning trees.
}\end{thm}

In 1982  Lov\'asz and Yemini~\cite{MR644960} showed that every $6$-connected graph  is $2$-rigid 
 and constructed a $5$-connected graph with   no   spanning minimally $2$-rigid  subgraphs.  
In 2005 Jord\'an~\cite{MR2171365}  extended  this result to a packing version by proving that
every  $6p$-connected graph has $p$ edge-disjoint spanning  $2$-rigid subgraphs.
In 2014 Cheriyan, Durand~de Gevigney, and Szigeti  established  the following  generalized version.
\begin{thm}{\rm (\cite{MR3171780})}\label{thm:Cheriyan,Gevigney,Szigeti}
{Every $(6p+2m)$-connected graph has a packing of $m$ spanning trees and  $p$ spanning $2$-rigid subgraphs.
}\end{thm}
Recently, Gu (2017)  formulated  the following extension of Theorem~\ref{thm:Cheriyan,Gevigney,Szigeti}
 and  used  it to refine a result on arc-connected orientation of graphs.
\begin{thm}{\rm (\cite{MR3619513})}\label{thm:Gu}
{Every $(4kp-2p+2m)$-connected graph with $k\ge 2$ has a packing of $m$ spanning trees 
and $p$  
 spanning 
$k$-rigid 
subgraphs.
}\end{thm}

In this paper, we  generalize and improve the above-mentioned  theorem to the following  supermodular version.
From this result,  we improve Theorem 1.4 in \cite{MR3619513} as mentioned in the abstract and 
 also refine the result of Gu (2017) on arc-connected orientations of graphs.
Moreover, we  investigate  spanning rigid subgraphs with small degrees on independent sets  and derive 
 that every $6k$-connected bipartite graph $G$ with one partite set  $A$ and $k\ge 1$ has a spanning $2$-connected subgraph $H$ such that for each $v\in A$,
$d_H(v)\le \lceil d_G(v)/k\rceil +2.$
\begin{thm}
{Let $G$ be a simple graph, let $l$ be a nonincreasing  intersecting supermodular  nonnegative   integer-valued function on subsets of $V(G)$, and let $p$ and $k$ be two positive  integers with $k\ge 2$.
If $G$ is  $(4kp-2p+2l)$-connected,
then it has a packing of a spanning $l$-partition-connected  subgraph  and $p$ spanning $k$-rigid  subgraphs.
}\end{thm}
%
%
%
%
%
%
%
%
%
%
%
%
%
%
%
%
%
\section{Basic tools}
In this section, we present some basic tools for working with sparse  and rigid graphs.
The first one shows that  minimal and maximal rigid subgraphs containing  two given vertices are unique, 
when  the original graph is  sparse and $c\le 2$.
In particular, maximal rigid subgraphs are edge-disjoint.
\begin{proposition}
{Let $F$ be an $\ell$-sparse  graph, where $\ell$ is a $c$-intersecting supermodular  weakly subadditive integer-valued function on 
subsets of $V(F)$. If   $F[A]$ and $F[B]$ are  two  $\ell$-rigid subgraphs and 
 $|A\cap B|\ge c$, then 
 both of graphs $F[A\cup B]$ and $F[A\cap B]$ are $\ell$-rigid.
}\end{proposition}
\begin{proof}
{Since $F$ is $\ell$-sparse, we must have 
$e_{F}(A\cap B) \le \sum_{v\in A\cap B}\ell(v)-\ell(A\cap B),$
which can conclude   that
$$e_{F}(A\cup B) \ge e_F(A)+e_F(B)-e_F(A\cap B) \ge 
 \sum_{v\in A}\ell(v)-\ell(A)+\sum_{v\in B}\ell(v)-\ell(B)-\sum_{v\in A\cap B}\ell(v)+\ell(A\cap B).$$
According to the assumption,  $\ell$ is supermodular on $A$ and $B$, and so
$$e_{F}(A\cup B)\ge \sum_{v\in A\cup B}\ell(v)+\ell(A\cap B)-\ell(A)-\ell(B)
 \ge
 \sum_{v\in A\cup B}\ell(v)-\ell(A\cup B).$$
Therefore, the equalities must be hold, which can imply that both of graphs $F[A\cup B]$ and  $F[A\cap B]$ are $\ell$-rigid. 
Hence  the proposition holds.
}\end{proof}
The next proposition 
is a useful tool for finding a pair of edges such that replacing them preserves
sparse property of a given spanning sparse subgraph.
\begin{proposition}\label{prop:xGy-exchange}
{Let $G$ be a graph and let $\ell$ be a $2$-intersecting supermodular weakly  subadditive integer-valued function on 
subsets of $V(G)$.
If $F$ is a spanning  $l$-sparse  subgraph  of $G$,  $xy\in E(G)\setminus E(F)$, and 
$Q$ is an
 $\ell$-rigid subgraph of $F$ including $x$ and $y$ with the minimum number of vertices,
 then for every  $e \in E(Q)$, the resulting graph $F-e+xy$ remains $\ell$-sparse. 
}\end{proposition}
\begin{proof}
{Suppose, by way of contradiction, that there is an edge $uv$ such that $F'=F-uv+xy$  is not $\ell$-sparse so that there is a vertex set $A$ with $e_{F'}(A)\ge \sum_{v\in A}\ell(v)-\ell(A)$.
Since $e_{F}(A)\le  \sum_{v\in A}\ell(v)-\ell(A)$, we must have $x,y \in A$, and  $A\setminus  \{u,v\}\neq \emptyset$, and also $e_{F}(A)=  \sum_{v\in A}\ell(v)-\ell(A)$. In other words, the graph $F[A]$ is $\ell$-rigid. 
Since $|V(Q)|$ is minimal and $A$ includes $x$ and $y$, one can conclude that $ V(Q)\subseteq A$.
This implies that $u,v\in  A$, which is a contradiction.
}\end{proof}
\begin{proposition}\label{prop:Q1Q2e}
{Let $F$ be an $\ell$-spars graph with $x, y \in V(F)$, where   $\ell$ is  a $2$-intersecting supermodular weakly  subadditive integer-valued function on  subsets of $V(F)$.
Let $F[A]$ be an $\ell$-rigid subgraph with the minimum number of vertices including $x$ and $y$.
If $xy\notin E(F)$ and  $B$ is  a vertex set including $x$ and $y$ such that  $F[B]+xy$ is $\ell$-rigid,
then   the graph $F[A\cup B]$ must be $\ell$-rigid.
}\end{proposition}
\begin{proof}
{If $A$ is a subset of $B$, then   since $F[B]/A$ is $\ell$-partition-connected, we have 
$$e_{F}( B) \ge e_F(A)+e_{F[B]}(P)\ge \sum_{v\in A}\ell(v)-\ell(A)+\sum_{X\in P}\ell(X)-\ell( B)=
\sum_{v\in  B}\ell(v)-\ell( B), $$
where $P$ is the partition of $ B$ with $P=\{A\} \cup \{\{v\}:v\in B\setminus A\}$.
Now,  assume that $|A\cap B|<|A|$.
Since $|A|$ is minimal, $e_F(A\cap B)< \sum_{v\in A\cap B}\ell(v)-\ell(A\cap B)$.
Since $F[B]+xy$ is $\ell$-rigid, we must have $e_F(B)= \sum_{v\in B}\ell(v)-\ell( B)-1$,
which can conclude   that
$$e_{F}(A\cup B) \ge e_F(A)+e_F(B)-e_F(A\cap B) \ge 
 \sum_{v\in A}\ell(v)-\ell(A)+\sum_{v\in B}\ell(v)-\ell(B)-\sum_{v\in A\cap B}\ell(v)+\ell(A\cap B).$$
According to the assumption,  $\ell$ is supermodular on $A$ and $B$, and so
$$e_{F}(A\cup B)\ge \sum_{v\in A\cup B}\ell(v)+\ell(A\cap B)-\ell(A)-\ell(B)
 \ge
 \sum_{v\in A\cup B}\ell(v)-\ell(A\cup B).$$
Therefore, in both cases  $F[A\cup B]$ must be   $\ell$-rigid. 
Hence  the proposition holds.
}\end{proof}

\begin{proposition}{\rm (\cite{P})}\label{prop:Q:subadditive}
{Let $F$ be a graph with $x,y\in V(F)$ and let   $\ell$ be  a  subadditive integer-valued function on 
subsets of $V(F)$.
If    $F$ is $\ell$-sparse and $Q$ is  an  $\ell$-rigid subgraph of $F$ with the minimum number of vertices including $x$ and $y$,
 then for every vertex set $A$ with $\{x,y\}\subseteq A\subsetneq V(Q)$, $d_Q(A)\ge 1$.
}\end{proposition}
\section{A sufficient connectedness condition for a graph to be $\ell$-rigid}
The following proposition establishes a necessary connectedness condition for a graph to be $\ell$-rigid.
\begin{proposition}
{Let $G$ be a graph and let $\ell$ be a weakly subadditive real function on subsets of $V(G)$.
If $G$ is $\ell$-rigid, then for any two disjoint vertex  sets $A$ and $B$,
$$d_{G-B}(A) \ge \ell(A\cup B)-\sum_{v\in B} \ell(v)\, +(\ell(G\setminus A)-\ell(G)).$$
}\end{proposition}
\begin{proof}
{We may assume that $G$ is minimally $\ell$-rigid.
Since $G$ is $\ell$-sparse, 
$e_G(A\cup B)\le \sum_{v\in A\cup B}\ell(v)-\ell(A\cup B)$ and 
$e_G(A^c)\le \sum_{v\in A^c}\ell(v)-\ell(A^c)$, where $A^c=V(G)\setminus A$.
It is not hard to verify that
$d_{G-B}(A)= |E(G)|-e_G(A\cup B)-e_G(A^c)+e_G(B).$
Therefore,
$$d_{G-B}(A)\ge 
\sum_{v\in V(G)}\ell(v)-\ell(G)-\sum_{v\in A\cup B}\ell(v)+\ell(A\cup B)-\sum_{v\in A^c}\ell(v)+\ell(A^c)+e_G(B),$$
which implies that
$$d_{G-B}(A)\ge \ell(A\cup B)-\sum_{v\in B} \ell(v) +\ell(A^c)-\ell(G)+e_G(B).$$
Hence the proposition is proved.
}\end{proof}
\begin{cor}{\rm(\cite{MR3619513})}\label{cor:essentially}
{Let $k$ be an integer with $k\ge 2$. If  $G$ is a $k$-rigid  graph of order at least three, then  it must be  $k$-edge-connected and essentially $(2k-1)$-edge-connected, and also  for each vertex $v$, the graph $G-v$ remains $(k-1)$-edge-connected.
}\end{cor}
The following theorem gives a sufficient connectedness condition for a graph to be $\ell$-rigid.
\begin{thm}\label{thm:sufficient:rigid}
{Let $G$ be a graph and let $\ell$ be a $2$-intersecting supermodular weakly subadditive nonnegative integer-valued function on subsets of $V(G)$.
If for each vertex $v$, $d_G(v)\ge 2\ell(v)$ and  for any two disjoint vertex  sets $A$ and $B$ with $A\cup B\subsetneq V(G)$
 and  $e_G(A\cup B)>\sum_{v\in A\cup B} \ell(v)-\ell(A\cup B)$,
$$d_{G-B}(A) \ge
2\ell(A\cup B)-\sum_{v\in B} \ell(v),$$
then $G$ has a  spanning $\ell$-rigid subgraph $H$  
excluding a  given arbitrary edge set of size at most $\ell(G)$.
}\end{thm}
\begin{proof}
{Let $E$ be an edge set of size at most $\ell(G)$.
Let $\mathcal{F}$ be a spanning $\ell$-sparse subgraph of $G\setminus E$ with the maximum size.
Define  $\mathcal{A}$ to be the collection  of all vertex sets of the  maximal $\ell$-rigid subgraphs of $\mathcal{F}$.
Suppose, by way of contradiction, that $V(G)\notin \mathcal{A}$.
Let  $\mathcal{A}_0$  be the collection  of all vertex sets  $X$ in $\mathcal{A}$
with $e_G(X)=\sum_{v\in X}\ell(v)-\ell(X)$.
Define $\mathcal{P}$ be the collection of all vertex sets  in $\mathcal{A}\setminus \mathcal{A}_0$
along with the vertex sets $\{v\}$ with  $v\in V(G)\setminus \cup_{X\in \mathcal{A}\setminus \mathcal{A}_0}X$.
For any $X\in \mathcal{P}$, 
define $X_B$ to be the set of all vertices $v$ which appears in at least two vertex sets of $\mathcal{P}$,
 and set $X_A=X\setminus X_B$.
It is easy to see that
$$
\sum_{X\in \mathcal{P}}\sum_{v\in X}\ell(v)
-\frac{1}{2}\sum_{X\in \mathcal{P}} \sum_{v\in X_B}\ell(v) \ge
 \sum_{v\in V(G)}\ell(v).
$$
Since $|E(\mathcal{F})|$ is maximal, 
for  $xy\in E(G)\setminus (E\cup E(\mathcal{F}))$ there must be   an $\ell$-rigid subgraph of $\mathcal{F}$  including $x$ and $y$ 
with vertex set  in $\mathcal{A}$.
For every $X\in \mathcal{A}_0$, we have   $e_G(X)=e_\mathcal{F}(X)$, and so
 for every  $xy \in E(G)\setminus (E\cup E(\mathcal{F}))$ there must be an 
$\ell$-rigid subgraph of $\mathcal{F}$  including $x$ and $y$  with vertex set  in $\mathcal{P}$.
Hence 
$$e_\mathcal{F}(\mathcal{P})=e_{G}(\mathcal{P}).$$
By the assumption,
$$
e_{G\setminus E}(\mathcal{P})\ge
\frac{1}{2}\sum_{X\in \mathcal{P}}d_{G- X_B}(X_A)-|E|
\ge   \sum_{X\in \mathcal{P}}\big(\ell(X)-\frac{1}{2}\sum_{v\in X_B}\ell(v)\big)-\ell(G).
$$
Therefore,
$$|E(\mathcal{F})|=
e_\mathcal{F}(\mathcal{P})
+\sum_{X\in \mathcal{P}}e_\mathcal{F}(X)= 
e_{G\setminus E}(\mathcal{P})
+\sum_{X\in \mathcal{P}}(\sum_{v\in X}\ell(v)-\ell(X))\ge
   \sum_{v\in V(G)} \ell(v)-\ell(G).$$
Hence $\mathcal{F}$ must be  $\ell$-rigid, a contradiction.
}\end{proof}
\begin{remark}
{In the above-mentioned theorem we could reduce  the needed lower bound by  $2\ell(G)-2|E|$ 
for any two disjoint  vertex sets $A$ and $B$ with $|A\cup B|=|V(G)|-1$, 
where $E$ is the give edge set of size at most $\ell(G)$.
In fact, there are at most one vertex set $X$ in $\mathcal{A}$ with $|X|=|V(G)|-1$, when $|V(G)|\ge 4$.
This refined version can imply Corollary 1.9 in~\cite{Gu(2018)}.
}\end{remark}
%
%
%
%
%
%
%
%
%
%
%
%
%
%
%
\section{A necessary and sufficient decomposition condition}

\label{sec:Decomposition}
By the result of Nash-Williams~\cite{MR0133253} and Tutte~\cite{MR0140438}, a graph is $m$-partition-connected
 if and only if it can be decomposed  into $m$ edge-disjoint spanning trees.
Recently, the present author generalized this result to the following supermodular version.
\begin{thm}{\rm (\cite{P})}\label{thm:main:partition-connected}
{Let $G$ be a graph and let $l_1, l_2,\ldots, l_m$ be $m$  intersecting supermodular subadditive integer-valued functions on subsets of $V(G)$.  
Then $G$ is $(l_1+\cdots +l_m)$-partition-connected if and only if it can be decomposed into $m$ edge-disjoint spanning subgraphs $H_1,\ldots, H_m$ such that every graph $H_i$ is $l_i$-partition-connected.
}\end{thm}
For a special case, Theorem~\ref{thm:main:partition-connected}  
can  be developed to a  rigid version  as the following result which is a generalized version of Theorem 5.2 in~\cite{Gu(2018)}.
However, an arbitrary $(\ell_1+\ell_2)$-rigid graph  may have  no spanning minimally $\ell_1$-rigid subgraphs,  
 see Figure~\ref{Figure:example}.
\begin{thm}\label{thm:rigid:decomposition}
{Let $G$ be a graph and  $\ell$ be a $2$-intersecting supermodular weakly subadditive  integer-valued function on subsets of $V(G)$.
Assume that for any two adjacent vertices $u$ and $v$, $\ell(u)+\ell(v)=\ell(\{u,v\})+1$.
Then $G$ is $p\ell$-rigid  if and only if it can be decomposed into $p$ edge-disjoint  spanning $\ell$-rigid subgraphs.
}\end{thm}
\begin{proof}
{We may assume that $G$ is minimally $p\ell$-rigid so that $G$ is $p\ell$-sparse and $|E(G)|=\sum_{v\in V(G)}p \ell(v)-p \ell(G)$.
The proof presented here is based on matroid theory. 
We use some properties of the matroid obtained from the union of $p$ copies of a matroid consists of  all  spanning $\ell$-sparse subgraphs of $G$. 
Let $\mathcal{F}_1,\ldots, \mathcal{F}_p$ be $p$ edge-disjoint spanning $\ell$-sparse  subgraphs of $G$ 
with the maximum  $|E(\mathcal{F})|$, where $\mathcal{F}=\mathcal{F}_1\cup \cdots \cup \mathcal{F}_p$.
By a theorem of Edmonds on the rank of matroids~\cite{MR0237366},  there is a spanning subgraph $H$ of $G$ with 
$$|E(\mathcal{F})| = p\, \rank_{\ell}(H)+|E(G)\setminus E(H)|,$$
where $\rank_{\ell}(H)$ denotes  the maximum of all $|E(F)|$ taken over all spanning $\ell$-sparse subgraphs $F$ of $H$. 
Take $F$ be a spanning $\ell$-sparse subgraph of $H$ with the maximum $|E(F)|$.
Let $\mathcal{P}$ be the collection of subsets of $V(F)$ obtained from the maximal $\ell$-rigid subgraphs of $F$.
By the property of $\ell$,  every edge $e\in E(F)$ itself  is an $\ell$-rigid subgraph of $F$ and so lies in a maximal $\ell$-rigid subgraph of $F$.
By the maximality of $F$,  both ends of every edge $e\in E(H) \setminus  E(F)$ must  lie in an $\ell$-rigid subgraph of $F$; 
otherwise we can add it to $F$ to obtain a new spanning sparse subgraphs with larger size.
Thus  $e_{\mathcal{F}}(\mathcal{P})\le |E(G)\setminus E(H)|$, 
and also
$$\rank_{\ell}(H)=|E(F)|=\sum_{X\in \mathcal{P}}e_{F}(X)=\sum_{X\in \mathcal{P}} (\sum_{v\in X}\ell(v)-\ell(X)),$$
which implies that
$$ \sum_{X\in \mathcal{P}} e_{\mathcal{F}}(X)+e_{\mathcal{F}}(\mathcal{P})
= |E(\mathcal{F})|=
p \sum_{X\in \mathcal{P}}(\sum_{v\in X}\ell(v)-\ell(X))+|E(G)\setminus E(H)|.$$
On the other hand,
$$|E(\mathcal{F})|=\sum_{X\in \mathcal{P}} e_{\mathcal{F}}(X)+e_{\mathcal{F}}(\mathcal{P})\le
\sum_{X\in \mathcal{P}}e_G(X)+|E(G)\setminus E(H)|\le 
 \sum_{X\in \mathcal{P}}(\sum_{v\in X}p\ell(v)-p\ell(X))+|E(G)\setminus E(H)|.$$
These inequalities can imply that for every $X\in \mathcal{P}$,  
 $e_\mathcal{F}(X)=e_{G}(X)$ and also $e_{\mathcal{F}}(X)=|E(G)\setminus E(H)|$.
Therefore, $|E(\mathcal{F})|=|E(G)|=\sum_{v\in V(G)}p \ell(v)-p \ell(G)$, 
and so for every $\mathcal{F}_i$, we must have 
$|E(\mathcal{F}_i)|=\sum_{v\in V(G) }\ell(v)-\ell(G)$.
Hence the theorem holds.
}\end{proof}
\begin{figure}[h]
{\centering
\includegraphics[scale=1]{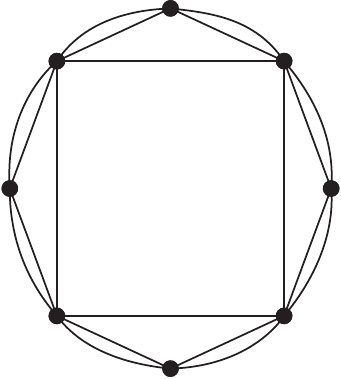}
\caption{An $\ell_{3,4}$-rigid graph with no spanning minimally  $\ell_{2,3}$-rigid subgraphs.}
\label{Figure:example}
}\end{figure}
%
%
%
%
%
%
%
%
%
%
%
%
\newpage
\section{Structures of  maximal  packing spanning sparse subgraphs}
Here, we state following fundamental theorem, 
which gives much information about maximal  packing spanning sparse subgraphs.
This result is a  supplement of a recent result in~\cite{P}  and provides another extension for Lemma~3.5.3 in~\cite{MR1743598}.
\begin{thm}\label{thm:generalized:D}
{Let $G$ be a graph, let $l$ be an  intersecting supermodular    subadditive  integer-valued function on subsets of $V(G)$, and let $\ell$ be a $2$-intersecting supermodular  subadditive   integer-valued function on subsets of $V(G)$.
 If $F$ and $\mathcal{F}$ are two edge-disjoint spanning subgraphs of $G$ 
with the maximum $|E(F \cup \mathcal{F})|$
such that $F$ is $l$-sparse and $\mathcal{F}$ is $\ell$-sparse,
 then there is a partition $P$ of $V(G)$  with the following properties:
\begin{enumerate}{
\item For every $A\in P$, the graph $F[A]$ is $l$-partition-connected.
\item There is no edges in $E(G)\setminus E(F\cup \mathcal{F})$ joining different parts of $P$.
\item   For every $xy\in E(G)\setminus E(\mathcal{F})$ with $x,y\in A\in P$,
 there is a vertex set $X$ such that  $\{x,y\} \subseteq X\subseteq A$ and $\mathcal{F}[X]$ is $\ell$-rigid.
}\end{enumerate}
}\end{thm}
\begin{proof}
{Define $\ell_1=l$, $\ell_2=\ell$, $F_1=F$, and $F_2=\mathcal{F}$. Put  $T=(F_1,F_2)$. 
 Let $\mathcal{A}$ be the set of all $2$-tuples $\mathcal{T}=(\mathcal{F}_1,\mathcal{F}_2)$ 
with the maximum  $|E(\mathcal{T})|$
 such that $\mathcal{F}_1$ and  $\mathcal{F}_2$ are edge-disjoint spanning subgraphs of $G$ 
and every $\mathcal{F}_i$ is $\ell_i$-sparse,
where $E(\mathcal{T}) =E(\mathcal{F}_1 \cup \mathcal{F}_2)$.
Note that if  $e\in E(G)\setminus E(\mathcal{T})$, then every graph $\mathcal{F}_i+e$ is not $\ell_i$-sparse; 
otherwise, we replace $\mathcal{F}_i$ by $\mathcal{F}_i+e$ in $\mathcal{T}$, 
which contradicts maximality of $|E(\mathcal{T})|$.
Thus both ends of $e$ lie in an $\ell_i$-rigid subgraph of $\mathcal{F}_i$.
Let $Q_i$ be the  $\ell_i$-rigid subgraph of $\mathcal{F}_i$ 
including both ends of $e$ 
with minimum number of vertices.
Let $e'\in Q_i$.
Define $\mathcal{F}'_i=\mathcal{F}_i-e'+e$, 
and $\mathcal{F}'_j=\mathcal{F}_j$ for other $j$ with $j\neq i$.
According to Proposition~\ref{prop:xGy-exchange}, 
the graph $\mathcal{F}'_i$ is again $\ell_i$-sparse and so 
$\mathcal{T}'=(\mathcal{F}'_1, \mathcal{F}'_2) \in \mathcal{A}$.
We say that $\mathcal{T}'$ is obtained from $\mathcal{T}$ by replacing a pair of edges.
Let $\mathcal{A}_0$ be the set of all $2$-tuples $\mathcal{T}$ in $\mathcal{A}$ 
which can be obtained from $T$ by a series of edge replacements.
Let $G_0$ be the spanning subgraph of $G$ with
$$E(G_0)=\bigcup_{\mathcal{T}\in  \mathcal{A}_0} (E(G)\setminus E(\mathcal{T})).$$
Now, we prove the following claim.
%
\vspace{2mm}\\
{\bf Claim.} 
Let  $\mathcal{T}=(\mathcal{F}_1, \mathcal{F}_2) \in \mathcal{A}_0$ and 
assume  that 
$\mathcal{T}'=(\mathcal{F}'_1, \mathcal{F}'_2)$  is obtained from $\mathcal{T}$ by replacing a pair of edges.
If $x$ and $y$ are two vertices in an  $\ell_i$-rigid subgraph  of $\mathcal{F}'_i \cap G_0$, 
then $x$ and $y$ are also  in an $\ell_i$-rigid subgraph of $\mathcal{F}_i \cap G_0$, where $i=1,2$.
\vspace{2mm}\\
{\bf Proof of Claim.} 
Let $e'$ be the new edge in $E(\mathcal{T}')\setminus E(\mathcal{T})$.
Define $Q'_i$ to be the minimal $\ell_i$-rigid subgraph  of $\mathcal{F}'_i \cap G_0$ including $x$ and $y$.
We may assume that $e'\in E(Q'_i)$; 
otherwise, $E(Q'_i) \subseteq E(\mathcal{F}_i) \cap E(G_0)$ and the proof can easily be completed.
Since $e' \in E(\mathcal{T}')\setminus E(\mathcal{T})$, both ends  of $e'$
  must lie in an $\ell_i$-rigid subgraph   of $\mathcal{F}_i $.
Define $Q_i$ to be the minimal $\ell_i$-rigid subgraph  of $\mathcal{F}_i$
including both ends of $e'$.
By Proposition~\ref{prop:xGy-exchange}, for every edge $e\in E(Q_i)$,
 the graph $\mathcal{F}_i-e+e'$ remains $\ell_i$-sparse, which can imply that $E(Q_i)\subseteq E(G_0)$.
Define $Q=(Q_i\cup Q'_i)-e'$.
Note that  $Q$ includes  $x$ and $y$, and also $E(Q)\subseteq E(G_0)\cap E(\mathcal{F}_i)$.
By Proposition~\ref{prop:Q1Q2e}, the graph $Q$ must be  $\ell_i$-rigid and so  the claim holds.

Define $P$ to be the partition of $V(G)$ obtained from the components of $G_0$.
Let $i \in \{1,2\}$,
let   $C_0$ be a  component of $G_0$, 
and let $xy\in E(C_0)$.  
By the definition of $G_0$,  
there is no edges in $E(G)\setminus E(F_1\cup F_2)$ joining different parts of $P$,
and also  there are some $2$-tuples $\mathcal{T}^1,\ldots, \mathcal{T}^n$  in
 $\mathcal{A}_0$ such that
 $xy \in E(G)\setminus E(\mathcal{T}^n)$, 
 $T= \mathcal{T}^1$,
and every $\mathcal{T}^k$
 can be obtained from $\mathcal{T}^{k-1}$ by replacing a pair of edges,  where $1 < k \le n$.
As we stated above,  $x$ and $y$ must lie   in an $\ell_i$-rigid  subgraph  of $\mathcal{F}^n_i$.
Let $Q'_i$ be the  minimal $\ell_i$-rigid subgraph of $\mathcal{F}^n_i$
 including $x$ and $y$.
By Proposition~\ref{prop:xGy-exchange}, 
for every edge $e\in E(Q'_i)$, the graph  $\mathcal{F}^n_i-e+xy$ remains $\ell_i$-sparse,  which can imply
$E(Q'_i)\subseteq E(G_0)$. 
Thus $x$ and $y$ must also lie  in an $\ell_i$-rigid subgraph  of $\mathcal{F}^n_i\cap G_0$.
By repeatedly applying the above-mentioned claim, 
 one can conclude  that 
$x$ and $y$ lie  in an  $\ell_i$-rigid subgraph of $F_i \cap G_0$.
Let $Q_i$ be the  minimal $\ell_i$-rigid subgraph of $F_i$  including $x$ and $y$
so that $E(Q_i)\subseteq E(G_0)$. 
Since $\ell_i$ is subadditive, 
Proposition~\ref{prop:Q:subadditive} implies that 
$d_{Q_i}(A)\ge 1$, for   every vertex set $A$ with $\{x,y\}\subseteq A\subsetneq V(Q_i)$.
Thus  $Q_i/\{x,y\}$ is connected and hence $V(Q_i)\subseteq V(C_0)$.
In other words, for every $xy\in E(C_0)$, there is an $\ell_i$-rigid subgraph of $F_i\cap C_0$ including $x$ and $y$.
Since $C_0$ is connected and $\ell_1$ is intersecting supermodular, all vertices of $C_0$ must lie in an $\ell_1$-partition-connected subgraph of $F_1\cap C_0$. 
Thus  $F [V(C_0)]$ itself must be   $\ell_1$-partition-connected and also the edge set of $F [V(C_0)]$ is a subset of $E(C_0)$.
For every edge  $xy\in E(F)$ with $x,y\in V(C_0)$, 
by the above-mentioned claim, there is a minimal  $\ell_2$-rigid subgraph $Q$ of $F_2\cap G_0$  including $x$ and $y$.
As we observed above, one can conclude that  $E(Q)\subseteq E(C_0)$.
Hence the proof is  completed.
}\end{proof}
%
%
%
%
%
%
%
%
%
%
%
%
%
\section{Packing   spanning  partition-connected  and  spanning rigid  subgraphs}
\label{sec:partition-rigid:graph}
The following theorem presents a sufficient connectedness condition 
for the existence of a packing consists of  a spanning  $l$-partition-connected subgraph and  a spanning $\ell$-rigid  subgraph.
\begin{thm}\label{thm:partition-connected-rigid}
{Let $G$ be a graph, let $l$ be a nonincreasing  intersecting supermodular  nonnegative   integer-valued function on subsets of $V(G)$, and let $\ell$ be a $2$-intersecting supermodular  subadditive nonnegative   integer-valued function on subsets of $V(G)$.
If for each vertex $v$, $d_G(v)\ge 2\ell(v)+2l(v)$ and  for any two disjoint vertex  sets $A$ and $B$ with $A\cup B\subsetneq V(G)$
 and  $e_G(A\cup B)>\sum_{v\in A\cup B} \ell(v)-\ell(A\cup B)$,
$$d_{G-B}(A) \ge
2\ell(A\cup B)-\sum_{v\in B} \ell(v)+
 \begin{cases}
0,	&\text{when $A=\emptyset$};\\
2l(A\cup B),	&\text{when $A\neq \emptyset$},
\end {cases}$$
then $G$ can be decomposed into a spanning $l$-partition-connected  subgraph  and a spanning $\ell$-rigid  subgraph and 
  also a given arbitrary edge set of size at most $l(G)+\ell(G)$.
}\end{thm}
\begin{proof}
{Let $E$ be an edge set of size at most $l(G)+\ell(G)$.
Let $F$ and $\mathcal{F}$ be two edge-disjoint spanning subgraphs of $G\setminus E$ with the maximum $|E(F)|+|E(\mathcal{F})|$
such that $F$ is $l$-sparse and $\mathcal{F}$ is $\ell$-sparse.
Let $P$ be a partition of $V(G)$ with the properties described in Theorem~\ref{thm:generalized:D}.
Define  $\mathcal{A}$ to be the collection  of all vertex sets of the  maximal $\ell$-rigid subgraphs of all graphs $\mathcal{F}[X]$, 
where $X\in P$.
We may assume that $V(G)\notin \mathcal{A}$.
Let  $\mathcal{A}_0$  be the collection  of all vertex sets  $X$ in $\mathcal{A}$
with $e_G(X)=\sum_{v\in X}\ell(v)-\ell(X)$.
Define $\mathcal{P}$ be the collection of all vertex sets  in $\mathcal{A}\setminus \mathcal{A}_0$
along with the vertex sets $\{v\}$ with  $v\in V(G)\setminus \cup_{X\in \mathcal{A}\setminus \mathcal{A}_0}X$.
For any $X\in \mathcal{P}$, 
define $X_B$ to be the set of all vertices $v$ which appears in at least two vertex sets of $\mathcal{P}$,
 and set $X_A=X\setminus X_B$.
It is easy to see that
\begin{equation}\label{eq:base}
\sum_{X\in \mathcal{P}}\sum_{v\in X}\ell(v)
-\frac{1}{2}\sum_{X\in \mathcal{P}} \sum_{v\in X_B}\ell(v) \ge
 \sum_{v\in V(G)}\ell(v).
\end{equation}
For every $X\in \mathcal{A}_0$, we have
$e_{G}(X) =e_\mathcal{F}(X) $, and so
 for every $xy\in E(F)$ with $x,y\in A\in P$,
there must be an $\ell$-rigid subgraph of $\mathcal{F}$  including $x$ and $y$  whose vertex set is a subset of $A$.
Thus  items (2) and (3) of Theorem~\ref{thm:generalized:D} can imply that
\begin{equation}\label{eq:simplegraph-application}
e_{G}(\mathcal{P})= e_{\mathcal{F}}(\mathcal{P}) + e_F(P).
\end{equation}
By the assumption,
$$e_{G\setminus E}(\mathcal{P})\ge
\frac{1}{2}\sum_{X\in \mathcal{P}}d_{G- X_B}(X_A)-|E|
\ge   \sum_{X\in \mathcal{P}}\big(\ell(X)-\frac{1}{2}\sum_{v\in X_B}\ell(v)\big)+ 
\sum_{ X\in \mathcal{P}, X_A\neq \emptyset}l(X)-l(G)-\ell(G).$$
Now, we prove the following claim.
%
\vspace{2mm}\\
{\bf Claim.} 
If $Q\in P$, then there is a vertex set $X\in \mathcal{P}$ with $X\subseteq Q$ and  $X_A\neq \emptyset$.
\vspace{2mm}\\
{\bf Proof of Claim.} 
Suppose, by way of contradiction, that every vertex $v$ of $Q$ appears in at least two vertex sets $X$ in $\mathcal{P}$ 
with $X\subseteq Q$.
Note that $\mathcal{F}[Q]$ is not $\ell$-rigid and also the  $\ell$-rigid subgraphs 
$\mathcal{F}[X]$ with $X\in \mathcal{P}$ and  $X\subseteq Q$ are edge-disjoint.
Thus 
$$\sum_{v\in Q}\ell(v)-\ell(Q)> 
e_{\mathcal{F}}(Q)\ge 
\sum_{X\in \mathcal{P}, X\subseteq Q}(\sum_{v\in X}\ell(v)-\ell(X))= 
\sum_{X\in \mathcal{P}, X\subseteq Q}(\frac{1}{2}\sum_{v\in X}\ell(v)+\frac{1}{2}\sum_{v\in X} \ell(v)-\ell(X)).$$
Since $\sum_{v\in X} \ell(v)\ge 2\ell(X)$, one can conclude that
$$\sum_{v\in Q}\ell(v)-\ell(Q)> \sum_{v\in Q}\ell(v)
+\sum_{X\in \mathcal{P}, X\subseteq Q}(\frac{1}{2}\sum_{v\in X} \ell(v)-\ell(X))\ge 
\sum_{v\in Q}\ell(v),$$   
which implies   $ \ell(Q)< 0$.  This is  a contradiction. Hence the claim holds.
\vspace{5mm}
\\
Since $l$ is nonincreasing and nonnegative, by the above-mentioned claim we must have
$$\sum_{X\in \mathcal{P}, X_A\neq \emptyset}l(X)\ge \sum_{Q\in P}l(Q),$$
which implies that
\begin{equation}\label{eq:P:R:1}
e_{G\setminus E}(\mathcal{P})\ge
   \sum_{X\in  \mathcal{P}}\ell(X)-\frac{1}{2}\sum_{X\in \mathcal{P}}\sum_{v\in X_B} \ell(v)+\sum_{Q\in P}l(Q)-l(G)-\ell(G).
\end{equation}
On the other hand,
\begin{equation}\label{eq:P:R:2}
|E(\mathcal{F})|=
e_{\mathcal{F}}(\mathcal{P}) +
\sum_{X\in \mathcal{P}}e_{\mathcal{F}}(X) = 
e_{G\setminus E}(\mathcal{P})-e_{F }(P)+
\sum_{X\in \mathcal{P}}(\sum_{v\in X}\ell(v)-\ell(X)).
\end{equation}
Also, 
\begin{equation}\label{eq:P:R:3}
|E(F)|=e_{F}(P)+
\sum_{Q\in P}e_{F}(Q)= 
e_{F}(P)+
\sum_{Q\in P}(\sum_{v\in Q} l(v)-l(Q))=e_F(P)+\sum_{v\in V(G)}l(v)-\sum_{Q\in P}l(Q).
\end{equation}
Therefore, Relations~(\ref{eq:base}),~(\ref{eq:simplegraph-application}),~(\ref{eq:P:R:1}),~(\ref{eq:P:R:2}), and~(\ref{eq:P:R:3}) can conclude that
$$|E(F)|+|E(\mathcal{F})|
\ge   \sum_{v\in V(G)}(l(v)+\ell(v))-l(G)-\ell(G).$$
Thus we must have $|E(F)|=\sum_{v\in V(G)}l(v)-l(G)$ and  $|E(\mathcal{F})|=\sum_{v\in V(G)}\ell(v)-\ell(G)$.  
Hence  $F$ is $l$-partition-connected and $\mathcal{F}$ is $\ell$-rigid and the proof is completed.
}\end{proof}
\begin{remark}
{In the above-mentioned theorem we could reduce  the needed lower bound by  $2l(G)+2\ell(G)-2|E|$ 
for any two disjoint  vertex sets $A$ and $B$ with $|A\cup B|=|V(G)|-1$, 
where $E$ is the give edge set of size at most $l(G)+\ell(G)$.
This refined version can imply Corollary 1.8 in~\cite{Gu(2018)}.
}\end{remark}
%
%
%
%
%
%
%
%
%
The following corollary is an application  of Theorem~\ref{thm:partition-connected-rigid} which can help us to impose a bound on degrees.
\begin{cor}\label{cor:partition-connected-rigid:degrees}
{Let $G$ be a graph, let $l$ be a nonincreasing  intersecting supermodular  nonnegative   integer-valued function on subsets of $V(G)$, and let $\ell$ be a $2$-intersecting supermodular  subadditive nonnegative   integer-valued function on subsets of $V(G)$.
If for each vertex $v$, $d_G(v)\ge 2\ell(v)+2l(v)$ and  for any two disjoint vertex  sets $A$ and $B$ with $A\cup B\subsetneq V(G)$
 and  $e_G(A\cup B)>\sum_{v\in A\cup B} \ell(v)-\ell(A\cup B)$,
$$d_{G-B}(A) \ge
2\ell(A\cup B)-\sum_{v\in B} \ell(v)+
 \begin{cases}
0,	&\text{when $A=\emptyset$};\\
2l(A\cup B),	&\text{when $A\neq \emptyset$}.
\end {cases}$$
then $G$ has  a spanning subgraph $H$ 
containing a packing of a spanning $l$-partition-connected subgraph and a
 spanning $l$-rigid subgraph such that for each vertex $v$, 
$$d_{H}(v) \le \lceil\frac{d_G(v)}{2} \rceil+l(v)+\ell(v).$$
}\end{cor}
\begin{proof}
{Define $l'(v)=\lfloor \frac{d_G(v)}{2}\rfloor-l(v)-\ell(v)$, for each vertex $v$
so that $d_G(v)\ge 2l'(v)+2l(v)+2\ell(v)$. Define $l'(A)=0$ for every vertex set $A$ with $|A|\ge 2$.
By applying Theorems~\ref{thm:partition-connected-rigid} and~\ref{thm:main:partition-connected}, 
the graph $G$
 can be decomposed into a spanning $l'$-partition-connected subgraph $H'$,
a spanning $l$-partition-connected subgraph $H_1$,
 and a spanning $\ell$-rigid subgraph $H_2$.
Define $H=H_1\cup H_2$. 
For each vertex $v$, we have   $d_{H'}(v)\ge l'(G-v)+l'(v)-l'(G)=l'(v)$.
This implies that $d_H(v) = d_G(v)-d_{H'}(v)\le \lceil\frac{d_G(v)}{2} \rceil+l(v)+\ell(v)$.
}\end{proof}
%
%
%
%
\subsection{Further improvements on  connectivity requirements}
In this subsection, 
we shall introduce another step toward  improving Theorem~\ref{thm:partition-connected-rigid} 
as the following stronger but more complicated version.
This result  improves the needed  connectivity requirements a little.
\begin{thm}\label{thm:epsilon:partition-connected-rigid}
{Let $G$ be a graph, let $l$ be a nonincreasing  intersecting supermodular  nonnegative   integer-valued function on subsets of $V(G)$, and let $\ell$ be a $2$-intersecting supermodular  subadditive nonnegative   integer-valued function on subsets of $V(G)$.
Define  $\lambda$ to be the minimum of all $|X|$ taken over all vertex sets $X$ with $e_G(X)> \sum_{v\in X}\ell(v)-\ell(X)$.
Take $\phi$  to be a nonincreasing   real function on subsets of $V(G)$ with  $0\le \phi \le 1$.
If for each vertex $v$, $d_G(v)\ge 2\ell(v)+2l(v)$ and  for any two disjoint vertex  sets $A$ and $B$ with $A\cup B\subsetneq V(G)$
 and  $e_G(A\cup B)>\sum_{v\in A\cup B} \ell(v)-\ell(A\cup B)$,
$$d_{G-B}(A) +\epsilon(A\cup B)\ge
2\ell(A\cup B)-\sum_{v\in B} \ell(v)+l(A\cup B)\times
 \begin{cases}
2,	&\text{when $B=\emptyset$};\\
\frac{1}{\lambda}\phi(A\cup B) ,	&\text{when $A=\emptyset$};\\
2-\phi(A\cup B),	&\text{when $A\neq \emptyset$ and $B\neq \emptyset$},
\end {cases}$$
then $G$ can be decomposed into a spanning $l$-partition-connected  subgraph  and a spanning $\ell$-rigid  subgraph and 
  a given arbitrary  edge set $M$ of size at most  $l(G)+\ell(G)$,
where
$\epsilon(X)=2l(G)+2\ell(G)-2|M| $ for every vertex set $X$ with $|X|=|V(G)|-1$; and  $\epsilon(X)=0$ otherwise.
}\end{thm}
\begin{proof}
{The proof follows with the same arguments of Theorem~\ref{thm:partition-connected-rigid} with only  minor modifications. 
In fact, if for a vertex set  $Q\in P$, there is only one proper vertex subset $X$ of $Q$ with  $X\in \mathcal{P}$ and $X_A\neq \emptyset$, then there are at least $\lambda$  proper vertex subsets $X$ of $Q$ with $X\in \mathcal{P}$ and  $X_A= \emptyset$.
}\end{proof}
\begin{cor}
{Let $G$ be a simple graph and let $k$ be an  integer with $k\ge 2$.
If $G$ is $4k$-edge-connected, and $G-B$ is $(4k-1-k|B|)$-edge-connected for every vertex set   $B$,
 then $G$ has a spanning tree $T$  such that  $G- E(T)$  is $k$-rigid.
}\end{cor}
\begin{proof}
{Since $G$ has no multiple edges, it is not hard to verify that  for every vertex set $X$ with $e_G(X)>  k|X|-(2k-1)$,  we must have $|X|\ge 2k$  which implies that $k|X|-(4k-2)\ge 2(k-1)^2\ge 1$.
Now, it is enough to apply Theorem~\ref{thm:epsilon:partition-connected-rigid} with $\ell=\ell_{k,2k-1}$, $l=l_{1,1}$, $\lambda=1$, and $\phi =1$.
}\end{proof}
\begin{cor}
{Let $G$ be a simple graph and let $k$ be an  integer with $k\ge 2$.
If $G$ is $(2k+2)$-edge-connected and essentially $4k$-edge-connected, 
and $G-B$ is essentially $(4k-1-k|B|)$-edge-connected for every vertex set   $B$,
 then $G$ has a spanning tree $T$  such that  $G- E(T)$  is $k$-rigid.
}\end{cor}
The next corollary improves Corollary 1.11 in~\cite{Gu(2018)} a little.
\begin{cor}
{Every $6$-connected  essentially $8$-edge-connected simple graph $G$ has a spanning tree $T$  such that  $G- E(T)$  is $2$-connected.
}\end{cor}
%
%
%
%
%
\section{A necessary and sufficient orientation condition for a graph to be  $\ell$-rigid}
In 1980 Frank formulated   the following  criterion for a graph to be $l$-partition-connected.
\begin{thm}{\rm (\cite{MR579073})}\label{thm:Frank}
{Let $G$ be a   graph and let $l$ 
be an intersecting  supermodular nonnegative integer-valued   function on  subsets of $V(G)$ with $l(\emptyset)=l(G)=0$.
Then $G$ is $l$-partition-connected if and only if it has an $l$-arc-connected orientation.
}\end{thm}
By applying a special case of  the above-mentioned theorem due to Hakimi~\cite{MR0180501},  we  generalize Frank's result to the following rigid version.
\begin{thm}\label{thm:rigid:orientation:generalized}
{Let $G$ be a graph and let $\ell$ be a  weakly  subadditive nonnegative   integer-valued function on subsets of $V(G)$
with  $\ell(\emptyset)=\ell(G)=0$.
Then  $G$ is minimally $\ell$-rigid if and only if it has an $\ell$-arc-connected orientation such that for each vertex $v$, $d^-_G(v)=\ell(v)$.
}\end{thm}
\begin{proof}
{First assume that $G$ has an $\ell$-arc-connected orientation such that for each vertex $v$, $d^-_G(v)=\ell(v)$.
Obviously, $|E(G)|=\sum_{v\in V(G)}d^-_G(v)=\sum_{v\in V(G)}\ell(v)$.
Furthermore, for every vertex set $A$, we have 
$$e_G(A) =
\sum_{v\in A}d^-_G(v)-d_G^-(A)=
\sum_{v\in A}\ell(v)- d_G^-(A)
\le  
\sum_{v\in A}\ell(v)- \ell(A).$$
Thus $G$ is $\ell$-sparse and hence  minimally $\ell$-rigid.
Now, assume that $G$ is minimally $\ell$-rigid.
Since $G$ is $\ell$-sparse and $\ell$ is nonnegative, for every vertex set $A$,
 $e_G(A)\le \sum_{v\in A}\ell(v)-\ell(A)\le \sum_{v\in A}\ell(v)$.
Since $|E(G)|=\sum_{v\in V(G)}\ell(v)$, the graph  $G$ must have an orientation such that for each vertex $v$, $d^-_G(v)= \ell(v)$, see~\cite[Theorem 4]{MR0180501}.
Therefore, for every vertex set $A$, we  must have 
$$d_G^-(A)=\sum_{v\in A}d^-_G(v)-e_G(A)=  \sum_{v\in A}\ell(v)-e_G(A) \ge \ell(A).$$
Thus the orientation of $G$ is $\ell$-arc-connected. Note that the equality holds only if $G[A]$ is $\ell$-rigid.
}\end{proof}
A combination of  Theorem~\ref{thm:partition-connected-rigid} and~\ref{thm:rigid:orientation:generalized},  
can  conclude the next result.
\begin{thm}
{Let $G$ be a graph, let $l$ be a nonincreasing  intersecting supermodular  nonnegative   integer-valued function on subsets of $V(G)$, and let $\ell$ be a $2$-intersecting supermodular  subadditive nonnegative   integer-valued function on subsets of $V(G)$.
Let $r_1$ and $r_2$ be two nonnegative integer-valued functions on $V(G)$ which  
 $l(G)=\sum_{v\in V(G)}r_1(v)$ and  $\ell(G)=\sum_{v\in V(G)}r_2(v)$, and also $r_2\le \ell$.
If for each vertex $v$, $d_G(v)\ge 2\ell(v)+2l(v)$ and  for any two disjoint vertex  sets $A$ and $B$ with $A\cup B\subsetneq V(G)$
 and  $e_G(A\cup B)>\sum_{v\in A\cup B} \ell(v)-\ell(A\cup B)$,
$$d_{G-B}(A) \ge
2\ell(A\cup B)-\sum_{v\in B} \ell(v)+
 \begin{cases}
0,	&\text{when $A=\emptyset$};\\
2l(A\cup B),	&\text{when $A\neq \emptyset$}.
\end {cases}$$
then $G$ has  an orientation along with two  edge-disjoint spanning subdigraphs $H_1$ and $H_2$
such that $H_1$ is $r_1$-rooted $l$-arc-connected,  $H_2$ is $r_2$-rooted $\ell$-arc-connected,
and for each vertex $v$, 
$d^-_{H_1}(v)=l(v)-r_1(v)$, $d^-_{H_2}(v)=\ell(v)-r_2(v)$, and 
$$d^+_{G}(v) \le \lceil\frac{d_G(v)}{2} \rceil.$$
Furthermore, for a given arbitrary  vertex $u$ the upper bound can be reduced to  $ \lfloor  \frac{d_G(u)}{2}\rfloor$.
}\end{thm}
\begin{proof}
{First assume that $r_1=r_2=0$.
 For each vertex $v$, define $l_0(v)=\lfloor d_G(v)/2  \rfloor -\ell(v)-l(v)$ and 
define $\ell_0(A)=0$ for every vertex set $A$ with $|A| \ge 2$.
By applying Theorems~\ref{thm:partition-connected-rigid} and~\ref{thm:main:partition-connected}, 
the graph $G$
 can be decomposed into a spanning $l_0$-partition-connected subgraph $H_0$,
a spanning $l$-partition-connected subgraph $H_1$,
 and a spanning minimally $\ell$-rigid subgraph $H_2$.
By Theorem~\ref{thm:rigid:orientation:generalized}, every $H_i$ has an $l_i$-arc-connected orientation, where $l_1=l$ and $l_2=\ell$.
Consider the  orientation of  $G$ obtained from these orientations.
For each vertex $v$, we must have 
$d^+_{G}(v)\le d_{G}(v) -\sum_{0\le i\le 2}d^-_{H_i}(v)  \le \lceil \frac{d_G(v)}{2}\rceil$.
In order to  prove general case, one can apply the  same arguments by  replacing the set functions $l-r_1$ and $\ell-r_2$, 
where $r_i(A)=\sum_{v\in V(G)}r_i(v)$ for every vertex set $A$.
Note that for reducing the upper bound for the vertex $u$, 
the proof can be obtained by repeating the proof of Theorem~\ref{thm:partition-connected-rigid} with  minor modifications.
}\end{proof}
\begin{cor}
{Let $G$ be a graph and let $\ell$ be a $2$-intersecting supermodular  subadditive nonnegative   integer-valued function on subsets of $V(G)$ and  $r$ be a nonnegative integer-valued function on $V(G)$ with $r\le \ell$ and 
  $\ell(G)=\sum_{v\in V(G)}r(v)$.
If $G$ is $\ell$-weakly $2\ell$-connected, then it  has  an orientation along with a
 spanning $r$-rooted $\ell$-arc-connected  subdigraph $H$ such that for each vertex $v$, 
$d^-_H(v)=\ell(v)-r(v)$ and 
$$d^+_{G}(v) \le \lceil\frac{d_G(v)}{2} \rceil.$$
Furthermore, for a given arbitrary  vertex $u$ the upper bound can be reduced to  $ \lfloor  \frac{d_G(u)}{2}\rfloor$.
}\end{cor}
%
%
%
%
%
%
%
%
%
%
%
%
%
\section{Spanning rigid  subgraphs with small degrees on independent sets}
In this section, we turn our attention to present the following  strengthened version of Theorem~\ref{thm:partition-connected-rigid} 
by restricting degrees. Note that this theorem can be refined to a more  complicated version similar to Theorem~\ref{thm:epsilon:partition-connected-rigid}.
\begin{thm}\label{thm:partition-connected-rigid:independent}
{Let $G$ be a graph, let $l$ be a nonincreasing  intersecting supermodular  nonnegative   integer-valued function on subsets of $V(G)$, and let $\ell$ be a $2$-intersecting supermodular  subadditive nonnegative   integer-valued function on subsets of $V(G)$.
Let $k$ be a real number  with $k> 2$ and let $\rho$ be a nonnegative  real function on $V(G)$ with  $\rho \le d_G$.
 If the following conditions hold:
\begin{enumerate}{
\item For every $S\subseteq V(G)$,   $e_G(S) \le \sum_{v\in S} \rho(v)+\frac{k}{k-2}(l(G)+\ell(G))$.
\item For each vertex $v$, $d_G(v)\ge k\ell(v)+kl(v)$ and  for any two disjoint vertex  sets $A$ and $B$ with $A\cup B\subsetneq V(G)$
 and  $e_G(A\cup B)>\sum_{v\in A\cup B} \ell(v)-\ell(A\cup B)$,
$$d_{G-B}(A) \ge
k\ell(A\cup B)-\frac{1}{2}\sum_{v\in B} k\ell(v)+
 \begin{cases}
0,	&\text{when $A=\emptyset$};\\
kl(A\cup B),	&\text{when $A\neq \emptyset$}.
\end {cases}$$
}\end{enumerate}
then $G$ has a spanning subgraph $H$ containing a packing of  a spanning $l$-partition-connected  subgraph  and a spanning $\ell$-rigid  subgraph  such that for each vertex $v$,
$$d_H(v)\le \big\lceil \frac{d_G(v)  -2 \rho(v)}{k}\big\rceil+\rho(v)+l(v)+\ell(v).$$
}\end{thm}
\begin{proof}
{We repeat the proof of  Theorem~\ref{thm:partition-connected-rigid} with some  modifications. 
By an argument  similar to the proof of Corollary~\ref{cor:partition-connected-rigid:degrees},  it is enough to show that $G$ has a packing of a spanning $l^\prime$-partition-connected  subgraph  and a spanning $\ell$-rigid  subgraph, where
 $l^\prime(v)= l(v)+\lfloor \frac{k-1}{k}d_G(v) -\frac{k-2}{k}\rho(v)\rfloor-l(v)-\ell(v)$  
for each vertex $v$,
and  $l^\prime(A)=l(A)$ for every  vertex set $A$ with $|A|\ge2$.
Let $F$ and $\mathcal{F}$ be two edge-disjoint spanning subgraphs of $G$
 with the maximum $|E(F)|+|E( \mathcal{F})|$ such that $F$ is $l^\prime$-sparse and $\mathcal{F}$ is $\ell$-sparse.

Let $P$  be a partition of $V(G)$ with the properties described in Theorem~\ref{thm:generalized:D}.
Define  $\mathcal{A}$ to be the collection  of all vertex sets of the  maximal $\ell$-rigid subgraphs of $\mathcal{F}[A]$, 
where $A\in P$.
We may assume that $V(G)\notin \mathcal{A}$.
Let  $\mathcal{A}_0$  be the collection  of all vertex sets  $X$ in $\mathcal{A}$
with $e_G(X)=\sum_{v\in X}\ell(v)-\ell(X)$.
Define $\mathcal{P}$ be the collection of all vertex sets  in $\mathcal{A}\setminus \mathcal{A}_0$
along with the vertex sets $\{v\}$ with  $v\in V(G)\setminus \cup_{X\in \mathcal{A}\setminus \mathcal{A}_0}X$.
For every $X\in \mathcal{A}_0$, we have
$e_{G}(X) =e_\mathcal{F}(X) $, and so
 for every $xy\in E(F)$ with $x,y\in A\in P$,
there must be an $\ell$-rigid subgraph of $\mathcal{F}$  including $x$ and $y$  whose vertex set lie in $\mathcal{P}$.
Thus  items (2) and (3) of Theorem~\ref{thm:generalized:D} can imply that
\begin{equation}\label{eq:simplegraph-application:res}
 e_{G}(\mathcal{P})= 
e_{\mathcal{F}}(\mathcal{P}) + e_F(P).
\end{equation}
Take $S$ to be the set of all vertices $v$ such that $\{v\}\in \mathcal{P}$, and  put 
$\mathcal{P}^\prime=\mathcal{P}\setminus \{\{v\}:v\in S\}$.
For any $X\in \mathcal{P}$, 
define $X_B$ to be the set of all vertices $v$ which appears in at least two vertex sets of $\mathcal{P}$,
 and set $X_A=X\setminus X_B$.
It is not hard to check that
$$\sum_{v\in S}d_G(v)-e_G(S)+e_{G\setminus S}(\mathcal{P}^\prime) 
=e_{G}(\mathcal{P})\ge
 \sum_{X\in \mathcal{P}^\prime}d_{G- X_B}(X_A)-e_{G\setminus S}(\mathcal{P}^\prime)+e_G(S),$$
which implies that
$$\frac{2}{k}e_{G\setminus S}(\mathcal{P}^\prime) \ge 
 \frac{1}{k}\sum_{X\in \mathcal{P}^\prime}d_{G- X_B}(X_A)-\frac{1}{k}  \sum_{v\in S} d_G(v)+\frac{2}{k}e_G(S).$$
Thus
$$
e_{G}(\mathcal{P}) =\sum_{v\in S}d_G(v)-e_G(S)+e_{G\setminus S}(\mathcal{P}^\prime)\ge 
 \frac{1}{k}\sum_{X\in \mathcal{P}^\prime}d_{G- X_B}(X_A)+\sum_{v\in S}\frac{k-1}{k}d_G(v)-\frac{k-2}{k}e_G(S).
$$
Since $e_G(S)\le \sum_{v\in S} \rho(v)+\frac{k}{k-2}(l(G)+\ell(G))$, we must have 
\begin{equation}\label{eq:I:R:1}
e_{G}(\mathcal{P}) \ge 
 \frac{1}{k}\sum_{X\in \mathcal{P}^\prime}d_{G- X_B}(X_A)+
\sum_{v\in S}(\frac{k-1}{k}d_G(v)-\frac{k-2}{k}\rho(v))-l(G)-\ell(G).
\end{equation}
By the assumption,
$$\frac{1}{k}\sum_{X\in \mathcal{P}^\prime}d_{G- X_B}(X_A)
\ge   \sum_{X\in \mathcal{P}^\prime}\big(\ell(X)-\frac{1}{2}\sum_{v\in X_B}\ell(v)\big)+
\sum_{ X\in \mathcal{P}^\prime, X_A\neq \emptyset}l(X),$$
which can  imply that
\begin{equation}\label{eq:I:R:2}
\frac{1}{k}\sum_{X\in \mathcal{P}^\prime}d_{G- X_B}(X_A))
\ge \sum_{X\in  \mathcal{P}}\ell(X)
-\frac{1}{2}\sum_{X\in \mathcal{P}}\sum_{v\in X_B} \ell(v)-\sum_{v\in S}\ell(v)+
\sum_{ X\in \mathcal{P}^\prime, X_A\neq \emptyset}l(X).
\end{equation}
Hence Relations~(\ref{eq:I:R:1}) and~(\ref{eq:I:R:2}) can deduce that
$$e_{G}(\mathcal{P})\ge 
  \sum_{X\in  \mathcal{P}}\ell(X)-\frac{1}{2}\sum_{X\in \mathcal{P}}\sum_{v\in X_B} \ell(v)
+\sum_{v \in S}l^\prime (v)+\sum_{ X\in \mathcal{P}^\prime, X_A\neq \emptyset}l^\prime(X)-l^\prime(G)-\ell(G).$$
Similar to the proof of Theorem~\ref{thm:partition-connected-rigid},   one can prove that that for any $Q\in P$, 
there is a vertex set $X$ in $\mathcal{P}$  with $X_A\neq \emptyset$.
Since $l^\prime$ is nonincreasing and nonnegative, we must have
$$\sum_{v \in S}l^\prime(v)+\sum_{X\in \mathcal{P}^\prime, X_A\neq \emptyset}l^\prime(X)=\sum_{X\in \mathcal{P}, X_A\neq \emptyset}l^\prime(X)
\ge \sum_{Q\in P}l^\prime(Q),$$
which can  imply that
\begin{equation}\label{eq:I:R:3}
e_{G}(\mathcal{P})\ge 
  \sum_{X\in  \mathcal{P}}\ell(X)-\frac{1}{2}\sum_{X\in \mathcal{P}}\sum_{v\in X_B} \ell(v)+
 \sum_{Q\in P}l^\prime(Q)-l^\prime(G)-\ell(G).
\end{equation}
On the other hand,
\begin{equation}\label{eq:I:R:4}
|E(\mathcal{F})|=
e_{\mathcal{F}}(P) +
\sum_{X\in \mathcal{P}}e_{\mathcal{F}}(X) \ge 
e_G(\mathcal{P})-e_{F }(P)+
\sum_{X\in \mathcal{P}}(\sum_{v\in X}\ell(v)-\ell(X)).
\end{equation}
Also, 
\begin{equation}\label{eq:I:R:5}
|E(F)|=e_{F}(P)+
\sum_{Q\in P}e_{F}(Q)= 
e_{F}(P)+
\sum_{Q\in P}(\sum_{v\in Q} l^\prime(v)-l^\prime(Q))=e_{F}(P)+\sum_{v\in V(G)}l^\prime(v)-\sum_{v\in Q}l^\prime(Q).
\end{equation}
Therefore, Relations~(\ref{eq:I:R:3}), ~(\ref{eq:I:R:4}), and~(\ref{eq:I:R:5}) can conclude that
$$|E(F)|+|E(\mathcal{F})|
\ge   \sum_{v\in V(G)} (l^\prime(v)+\ell(v))-l^\prime(G)+\ell(G).$$
Thus we must have $|E(F)|=\sum_{v\in V(G)}l^\prime(v)-l^\prime(G)$ and  $|E(\mathcal{F})|=\sum_{v\in V(G)}\ell(v)-\ell(G)$. 
Hence  $F$ is $l^\prime$-partition-connected and $\mathcal{F}$ is $\ell$-rigid and the proof is completed.
}\end{proof}
\begin{cor}
{Let $G$ be a bipartite graph with one partite set $A$ and let $k$ be a real number with $k\ge 1$.
If $G$ is $6k$-connected, then it has a spanning $2$-rigid subgraph $H$ such that for each $v\in A$,
$$d_H(v)\le \lceil \frac{d_G(v)}{k}\rceil +2.$$
}\end{cor}
\begin{proof}
{For each $v\in A$, define $\rho(v)=0$, and for each $v\in V(G)\setminus A$, define $\rho(v)=d_G(v)$. 
Now, it is enough to  apply
 Theorem~\ref{thm:partition-connected-rigid:independent} with $\ell=\ell_{2,3}$ and use the fact that every $2$-rigid graph is $2$-connected.
}\end{proof}
%
%
%
%
%
%
%
%
%
%
%
%
%
\section{Hypergraph versions}
 Let $\mathcal{H}$ be a hypergraph (possibly with repetition of hyperedges).
The vertex set and  the hyperedge set  of $\mathcal{H}$ are denoted by $V(\mathcal{H})$ and $E(\mathcal{H})$, respectively. 
The (co-rank) rank  of  $\mathcal{H}$  is the (minimum) maximum   size of its hyperedges.
The degree $d_\mathcal{H}(v)$ of a vertex $v$ is the number of hyperedges of $\mathcal{H}$ including  $v$.
For a set $X\subseteq V(\mathcal{H})$, 
we denote by $\mathcal{H}[X]$ the induced  sub-hypergraph of $\mathcal{H}$  with the vertex set $X$  containing
precisely those hyperedges  $Z$
of $\mathcal{H}$ with $Z\subseteq X$.
A spanning sub-hypergraph $F$ is called {\bf $l$-sparse}, if for all vertex sets $A$, $e_F(A)\le \sum_{v\in A} l(v)-l(A)$,
where $e_F(A)$ denotes the number of hyperedges $Z$ of $F$  with $Z\subseteq A$.
Likewise,  the hypergraph $\mathcal{H}$ is called {\bf $l$-partition-connected}, 
if for every partition $P$ of $V(\mathcal{H})$, 
$e_\mathcal{H}(P)\ge \sum_{A\in P}l(A)-l(\mathcal{H})$, 
where $e_\mathcal{H}(P)$ denotes the number of hyperedges of $\mathcal{H}$ joining different parts of $P$.
We say that a hypergraph $\mathcal{H}$ is {\bf $l$-rigid}, if 
  it contains a spanning $l$-sparse sub-hypergraph $F$ with $|E(F)|=\sum_{v\in V(F)}l(v)-l(F)$.
We call a hypergraph $\mathcal{H}$ directed, if for every hyperedge $Z$, a head vertex $u$ in $Z$ is specified; 
 other vertices of $Z-u$ are called the tails of $Z$.
For a vertex $v$, we denoted by $d^-_{\mathcal{H}}(v)$ the number of hyperedges with head  $v$ and denote by $d^+_{\mathcal{H}}(v)$ and the number of hyperedges with tail $v$. 
We say that a directed hypergraph   $\mathcal{H}$ is {\bf $l$-arc-connected}, if for every vertex set $A$, 
$d^-_{\mathcal{H}}(A)\ge l(A)$,
where $d^-_{\mathcal{H}}(A)$ denotes the number of hyperedges  $Z$  with head vertex in $A$ and 
$Z\setminus A\neq \emptyset $.
Likewise, $\mathcal{H}$ is called {\bf $r$-rooted $l$-arc-connected},
 if for every vertex set $A$, $d_\mathcal{H}^-(A)\ge l(A)-\sum_{v\in A}r(v)$, where
 $r$ is a nonnegative integer-valued on $V(\mathcal{H})$ with $l(\mathcal{H})=\sum_{v\in V(\mathcal{H})}r(v)$.
We denote by $d_{\mathcal{H}\ominus B}(A)$ the number of hyperedges $Z$ 
with $Z\cap A\neq \emptyset$ and $Z\setminus (A\cup B)\neq\emptyset$, where $A$ and $B$ are two disjoint vertex sets.
\subsection{A necessary and sufficient orientation condition for a hypergraph to be  $\ell$-rigid}
The following theorem is a hypergraph version of Theorem~\ref{thm:Frank} which was proved by  Frank, Kir\'aly, and Kir\'aly  (2003).
\begin{thm}{\rm (\cite{MR2021108})}
{Let $\mathcal{H}$ be a   hypergraph and let $l$ 
be an intersecting  supermodular nonnegative integer-valued   function on  subsets of $V(\mathcal{H})$
 with $l(\emptyset)=l(\mathcal{H})=0$.
Then $\mathcal{H}$ is $l$-partition-connected if and only if it has an $l$-arc-connected orientation.
}\end{thm}
Motivated by the above-mentioned theorem, we   state the following    hypergraph version of 
Theorem~\ref{thm:rigid:orientation:generalized}.
\begin{thm}
{Let $\mathcal{H}$ be a hypergraph and let $\ell$ be a  weakly  subadditive nonnegative   integer-valued function on subsets of $V(\mathcal{H})$
with  $\ell(\emptyset)=\ell(\mathcal{H})=0$.
Then  $\mathcal{H}$ is minimally $\ell$-rigid if and only if it has an $\ell$-arc-connected orientation such that for each vertex $v$, $d^-_\mathcal{H}(v)=\ell(v)$.
}\end{thm}
\begin{proof}
{First assume that $\mathcal{H}$ has an $\ell$-arc-connected orientation such that for each vertex $v$,
 $d^-_\mathcal{H}(v)=\ell(v)$.
Obviously, $|E(\mathcal{H})|=\sum_{v\in V(\mathcal{H})}d^-_\mathcal{H}(v)
=\sum_{v\in V(\mathcal{H})}\ell(v)$.
Furthermore, for every vertex set $A$, we have 
$$e_\mathcal{H}(A) =
\sum_{v\in A}d^-_\mathcal{H}(v)-d_\mathcal{H}^-(A)=
\sum_{v\in A}\ell(v)- d_\mathcal{H}^-(A)
\le  
\sum_{v\in A}\ell(v)- \ell(A).$$
Thus $\mathcal{H}$ is $\ell$-sparse and hence  minimally $\ell$-rigid.
Now, assume that $\mathcal{H}$ is minimally $\ell$-rigid.
Since $\mathcal{H}$ is $\ell$-sparse and $\ell$ is nonnegative, for every vertex set $A$,
 $e_\mathcal{H}(A)\le \sum_{v\in A}\ell(v)-\ell(A)\le \sum_{v\in A}\ell(v)$.
Since $|E(\mathcal{H})|=\sum_{v\in V(\mathcal{H})}\ell(v)$, the hypergraph  $\mathcal{H}$ must have an orientation such that for each vertex $v$, $d^-_\mathcal{H}(v)= \ell(v)$, see~\cite[Lemma 3.3]{MR2021108}.
Therefore, for every vertex set $A$, we  must have 
$$d_\mathcal{H}^-(A)=\sum_{v\in A}d^-_\mathcal{H}(v)-e_\mathcal{H}(A)=  \sum_{v\in A}\ell(v)-e_\mathcal{H}(A) \ge \ell(A).$$
Thus the orientation of $\mathcal{H}$ is $\ell$-arc-connected.
}\end{proof}
\subsection{Generalizations}
In this subsection, we only state the hypergraphs versions of the main results of this paper, which their proofs follow with the same arguments that  stated for whose graph versions.

\begin{proposition}
{Let $\mathcal{H}$ be a hypergraph and let $\ell$ be a weakly subadditive real function on subsets of $V(\mathcal{H})$.
If $\mathcal{H}$ is $\ell$-rigid, then for any two disjoint vertex  sets $A$ and $B$,
$$d_{\mathcal{H}\ominus B}(A)
 \ge \ell(A\cup B)-\sum_{v\in B} \ell(v)\, +(\ell(\mathcal{H}\setminus A)-\ell(\mathcal{H})).$$
}\end{proposition}

\begin{thm}
{Let $\mathcal{H}$ be a hypergraph with the co-rank $c$, $c\ge 2$, let $l$ be an  intersecting supermodular   subadditive   integer-valued function on subsets of $V(\mathcal{H})$, and let $\ell$ be a $c$-intersecting supermodular  subadditive   integer-valued function on subsets of $V(\mathcal{H})$.
 If $F$ and $\mathcal{F}$ are two edge-disjoint spanning sub-hypergraphs of $\mathcal{H}$ 
with the maximum $|E(F \cup \mathcal{F})|$
such that $F$ is $l$-sparse and $\mathcal{F}$ is $\ell$-sparse,
 then there is a partition $P$ of $V(\mathcal{H})$  with the following properties:
\begin{enumerate}{
\item For any $A\in P$, the hypergraph $F[A]$ is $l$-partition-connected.
\item There is no hyperedges in $E(\mathcal{H})\setminus E(F\cup \mathcal{F})$ joining different parts of $P$.
\item   For every $Z\in E(\mathcal{F})$ with $Z\subseteq  A\in P$,
 there is a vertex set $X$ such that  $Z \subseteq X\subseteq A$ and $\mathcal{F}[X]$ is $\ell$-rigid.
}\end{enumerate}
}\end{thm}
\begin{thm}
{Let $\mathcal{H}$ be a hypergraph with the rank $r$ and co-rank $c$, $c \ge 2$,  let $l$ be a nonincreasing  intersecting supermodular  nonnegative   integer-valued function on subsets of $V(\mathcal{H})$, and let $\ell$ be a $c$-intersecting supermodular  subadditive nonnegative   integer-valued function on subsets of $V(\mathcal{H})$.
 If for each vertex $v$, $d_\mathcal{H}(v)\ge r\ell(v)+rl(v)$ and for any two disjoint vertex  sets $A$ and $B$ with $A\cup  B\subsetneq V(\mathcal{H})$  and $e_{\mathcal{H}}(A\cup B)> \sum_{v\in A\cup B}\ell(v)-\ell(A\cup B)$, 
$$d_{\mathcal{H}\ominus B}(A) \ge
r\ell(A\cup B)-\frac{r}{2}\sum_{v\in B} \ell(v)+
 \begin{cases}
0,	&\text{when $A=\emptyset$};\\
r l(A\cup B),	&\text{when $A\neq \emptyset$},
\end {cases}$$
then $\mathcal{H}$ can be decomposed into a spanning $l$-partition-connected  sub-hypergraph  and a spanning $\ell$-rigid  sub-hypergraph, and also a given edge set of size at most $l(\mathcal{H})+\ell(\mathcal{H})$.
}\end{thm}
\begin{cor}
{Let $\mathcal{H}$ be a hypergraph with the rank $r$ and co-rank $c$, $c \ge 2$,  let $l$ be a nonincreasing  intersecting supermodular  nonnegative   integer-valued function on subsets of $V(\mathcal{H})$, and let $\ell$ be a $c$-intersecting supermodular  subadditive nonnegative   integer-valued function on subsets of $V(\mathcal{H})$ with $l(\mathcal{H})=\ell(\mathcal{H})=0$.
 If  for each vertex $v$, $d_\mathcal{H}(v)\ge r\ell(v)+rl(v)$ and for any two disjoint vertex  sets $A$ and $B$ with $A\cup  B\subsetneq V(\mathcal{H})$  and $e_{\mathcal{H}}(A\cup B)> \sum_{v\in A\cup B}\ell(v)-\ell(A\cup B)$, 
$$d_{\mathcal{H}\ominus B}(A) \ge
r\ell(A\cup B)-\frac{r}{2}\sum_{v\in B} \ell(v)+
 \begin{cases}
0,	&\text{when $A=\emptyset$};\\
r l(A\cup B),	&\text{when $A\neq \emptyset$},
\end {cases}$$
then $\mathcal{H}$ has  an orientation along with two edge-disjoint  spanning sub-hypergraphs $H_1$ and $H_2$ 
such that $H_1$ is $l$-arc-connected, $H_2$ is $\ell$-arc-connected,
and  for each vertex $v$, 
$d^-_{H_1}(v)=l(v)$, $d^-_{H_2}(v)=\ell(v)$, and 
$$d^+_{\mathcal{H}}(v) \le \lceil\frac{r-1}{r}d_\mathcal{H}(v) \rceil.$$
Furthermore, for a given arbitrary  vertex $u$ the upper bound can be reduced to  $ \lfloor  \frac{r-1}{r} d_\mathcal{H}(u)\rfloor$.
}\end{cor}
%
%
%
%
%
%
%
%
%
%
%
%
%
%
%
%
\section{Applications}
\label{subsec:application:degree}
The following  theorem improves  Theorem~4.1 in~\cite{MR3619513} by imposing a bound on degrees.
\begin{thm}\label{thm:app:rigid}
{Every $k$-weakly $(4kp-2p+2m)$-connected simple graph $G$ with  $k\ge 2$ has a spanning subgraph $H$ containing  a  packing  of $m$ spanning  trees and $p$ spanning $k$-rigid   subgraphs such that for each vertex $v$, 
 $$d_H(v)\le \lceil \frac{d_G(v)}{2}\rceil +kp+m.$$
}\end{thm}
\begin{proof}
{Apply   Corollary~\ref{cor:partition-connected-rigid:degrees}
 with $\ell=p\ell_{k,2k-1}$ and $l=l_{m,m}$, 
and  next apply Theorems~\ref{thm:main:partition-connected} and~\ref{thm:rigid:decomposition}.
}\end{proof}
The next  result  improves  Theorem~1.4 in~\cite{MR3619513}
 by  replacing  essentially edge-connectivity by  edge-connectivity.
\begin{thm}\label{thm:application:base}
{Every $k$-weakly $(4kp-2p+2m)$-connected simple graph $G$ with  $k\ge 2$ has a spanning subgraph $H$ containing  a  packing  of $m$ spanning  trees and $p$ spanning $k$-rigid $(2k-1)$-edge-connected  subgraphs such that for each vertex $v$, 
 $$d_H(v)\le \lceil \frac{d_G(v)}{2}\rceil +2kp-p+m.$$
}\end{thm}
\begin{proof}
{By applying   Corollary~\ref{cor:partition-connected-rigid:degrees}
 with $\ell=p\ell_{k,2k-1}$ and $l=l_{kp-p+m,m}$, 
and  next applying Theorems~\ref{thm:main:partition-connected} and~\ref{thm:rigid:decomposition}, 
 one can conclude that $G$  has a spanning subgraph $H$ containing  a packing of 
 $p$ spanning $k$-rigid subgraphs $G_1,\ldots, G_p$, 
and $p$ spanning $l_{k-1,0}$-partition-connected subgraphs  $G'_1,\ldots, G'_p$, 
and  also $m$ spanning trees $T_1,\ldots, T_m$
such that for each vertex $v$, 
$d_H(v)\le \lceil \frac{d_G(v)}{2}\rceil +2kp-p+m$.
By Corollary~\ref{cor:essentially}, every graph $G_i$ is $k$-edge-connected and  essentially $(2k-1)$-edge-connected.
Define $H_i=G_i\cup G'_i$. 
Since 
$\delta(H_i) \ge \delta(G_i) +\delta(G'_i)\ge k+(k-1)=2k-1,$
the graph $H_i$ must be $(2k-1)$-edge-connected.
Now, it is enough to consider  the graphs $H_1,\ldots, H_p$ and $T_1,\ldots, T_m$ as the packing of $H$ with the desired properties. Note that $G$ could have multiple edges with multiplicity at most $p$.
}\end{proof}
%
%
\subsection{$2$-connected $(2k-1)$-edge-connected $\{r-3,r-1\}$-factors}

Recently, the present author~\cite{II} showed that
every $(2\lceil r/6\rceil+2k)$-edge-connected  $r$-regular graph of even order with $r\ge 4$ has a 
$k$-tree-connected $\{r-3,r-1\}$-factor.
In the following, we  improve  this result for   highly connected graphs.
Before doing so, we  recall   the following lemma.
\begin{lem}{\rm(\cite{II})}\label{lem:odd-forest}
{Every $m$-tree-connected graph $G$ has a spanning forest $F$ with odd degrees
such that  for each vertex $v$, 
 $d_F(v)\le \lceil\frac{d_G(v)}{m}\rceil$.
}\end{lem}
\begin{thm}
{Every $(2\lceil r/6\rceil+4k-2)$-connected  $r$-regular graph of even order with $r\ge 4$ has a $k$-rigid $\{r-3,r-1\}$-factor.
}\end{thm}
\begin{proof}
{Put $m=\lceil r/6\rceil$.
By Theorem~\ref{thm:app:rigid}, 
the graph $G$ contains two edge-disjoint spanning subgraphs $L'$ and $L$ such that $L'$ is $k$-rigid, 
$L$ is $m$-tree-connected, and also  for each vertex $v$, 
 $d_{L}(v)+d_{L'}(v)\le \lceil\frac{d_G(v)}{2}\rceil+k+m$.
Note that   for each vertex $v$, we must have $d_{L}(v) \le \lceil\frac{d_G(v)}{2}\rceil+m$.
By Lemma~\ref{lem:odd-forest}, the graph $L$ has a spanning forest  $F$  with odd degrees such that for each vertex $v$,
$d_F(v)\le \lceil\frac{d_{L}(v)}{m}\rceil\le \lceil\frac{d_{G}(v)}{2m}\rceil +1= 4$.
 It is not hard to check that $G\setminus E(F)$ is the desired spanning subgraph we are looking for.
}\end{proof}
%
%
%
%
\subsection{Arc-connected orientations of graphs}
Recently, Gu~\cite{MR3619513} showed that 
every  $(2k+1)$-weakly $(8k+4)$-connected simple graph has an orientation such that for each vertex $v$, 
$G-v$  remains $k$-arc-strong. 
In the following, we  strengthen  this result in the  same way by replacing   a special case of 
Theorem~\ref{thm:application:base}.
For this purpose, we first recall the following lemma due to  Kir\'aly and Szigeti (2006).
\begin{lem}{\rm (\cite{MR2236505})}\label{lem:Eulerian-Orientation}
{An Eulerian graph $G$  has a smooth orientation such that for each vertex $v$, 
the resulting directed graph $G-v$ is $k$-arc-strong, if and only if 
for each vertex $v$, the graph $G-v$ is $2k$-edge-connected.
}\end{lem}
The following theorem improves Theorem 1.7 in~\cite{MR3619513}.
\begin{thm}\label{thm:arc-connected}
{Every  $(2k+1)$-weakly $(8k+4)$-connected simple graph $G$  has a $(2k+1)$-arc-strong smooth orientation such that for each vertex $v$, 
the resulting directed graph $G-v$  remains $k$-arc-strong.
}\end{thm}
\begin{proof}
{By applying   Theorem~\ref{thm:application:base} with $p= m=1$ and replacing $2k+1$ instead of $k$, the graph 
 $G$ can be decomposed into a  spanning tree $T$ and a spanning $(2k+1)$-rigid $(4k+1)$-edge-connected subgraph $G'$.
According to Corollary~\ref{cor:essentially}, for each vertex $v$ the graph $G'-v$ remains $2k$-edge-connected.
It is not hard  to check that there is a spanning forest $F$ of $T$ 
such that for each vertex $v$, $d_F(v)$ and $d_{G'}(v)$ have the same parity.
Define $H$ to be the spanning Eulerian subgraph of $G$ with $E(H)=E(G')\cup E(F)$.
Note that $H$ must automatically be $(4k+2)$-edge-connected.
By Lemma~\ref{lem:Eulerian-Orientation}, the graph $H$   has a smooth orientation such that for each vertex $v$,  the resulting directed graph $H-v$  remains $k$-arc-strong. Since  this orientation is Eulerian, it is also  $(2k+1)$-arc-strong.
Now, it enough to consider a smooth orientation for the spanning  graph $H'$    of $G$  with $E(H')=E(G)\setminus E(H)$ and 
induce whose orientation to $G$.
This can complete the proof.
}\end{proof}
%
%
%
%
%
%
%
%
%
%
%
%
%
%
%
%
%
%

\end{document}